\documentclass{amsart}
\usepackage{amssymb}
\usepackage{hyperref} 

\begin{document} 
\title[Global invertibility and applications]
{Global invertibility of mappings between Banach spaces and
applications to nonlinear equations}

\author[M. Galewski and D. Repov\v{s}]
{Marek Galewski and Du\v{s}an Repov\v{s}}

\address{Marek Galewski \newline
 Institute of Mathematics, Lodz University of
Technology, Wolczanska 215, 90-024 Lodz, Poland}
\email{marek.galewski@p.lodz.pl}

\address{Du\v{s}an Repov\v{s} \newline
Faculty of Education
and
 Faculty of Mathematics and Physics,
University of Ljubljana, SI-1000, Slovenia}
\email{dusan.repovs@guest.arnes.si}

\keywords{Global diffeomorphism; local diffeomorphism; algebraic equation;
\hfill\break\indent mountain pass lemma; integro-differential system}

\begin{abstract}
 We provide sufficient conditions for a mapping between two Banach
 spaces to be a diffeomorphism using the approach of an auxiliary functional
 and also by the aid of a duality mapping corresponding to a normalization
 function. We simplify and generalize our previous results. Applications to
 algebraic equations  and to integro-differential systems are also given.
\end{abstract}

\maketitle
\numberwithin{equation}{section}
\newtheorem{theorem}{Theorem}[section]
\newtheorem{lemma}[theorem]{Lemma}
\newtheorem{proposition}[theorem]{Proposition}
\newtheorem{remark}[theorem]{Remark}
\newtheorem{corollary}[theorem]{Corollary}
\newtheorem{example}[theorem]{Example}
\allowdisplaybreaks

\section{Introduction}

The aim of this work is to present a scheme allowing for a mapping between
two Banach spaces which defines a local diffeomorphism, to become a global one.
The methods applied are those of critical point theory. The main idea behind
the approach presented here is already well known, i.e.\ adding some
conditions which make a local diffeomorphism (i.e.\ invertibility around each
point) a global one (i.e.\ a mapping between whole spaces). It seems this was
started by Hadamard \cite{Hadamard} and later developed by several
authors in a number of works, out of which we single out the
following:
\begin{itemize}
\item[(i)] \cite{zampieri} for some overview of invertibility results and some
general up-to-date version of the Hadamard-L\'{e}vy Theorem
(see \cite{Levy,plastock});

\item[(ii)] \cite{ioffee} for a general result by which we compare our
approach in \cite{galewskiRadulescu}, showing that these two do not overlap;

\item[(iii)] \cite{fijalkowski2} for a development of higher-order invertibility
conditions;

\item[(iv)] \cite{katriel} for the metric space setting;

\item[(v)] \cite{rad} for some invertibility on a finite-dimensional space.
\end{itemize}

Our inspiration lies in the works \cite{SIW,idczakIFT,IDCZAK_Gen_GIFT}
 which we somewhat improve and generalize while retaining
their main methodology. Although we use known approaches and techniques, we
put emphasis on presenting the scheme, which while being intuitive in the
case of functions of one variable, becomes more involved when the dimension
is enlarged to certain $n\geq 2$ and even next, to an infinite setting. We
underline that the simple remarks working on the real line can be suitably
extended to cover other cases of domains (finite and infinite). In this
work we will concentrate on the case of mappings between Banach
spaces and our aim is to make, with the aid of some additional assumption,
the local invertibility condition a global one, using some already developed
tools. Thus we will somewhat simplify known proofs and provide some insight into
understanding of the tools used, and  introduce new notions and
approaches. We still believe it is of some importance to collect in one work
certain results and present them in a more uniform manner.

This article is organized as follows. In Section 2 we remark on the
scheme leading to producing a global diffeomorphism out of a local one in
the spirit of the Hadamard Theorem. In Section 3 we provide necessary
background with which we can proceed to the infinite-dimensional space setting.
We present two approaches leading to global diffeomorphism between two Banach
spaces with the aid of an auxiliary functional and a duality mapping. In the
last section we provide some applications to algebraic equations, first in a
finite-dimensional setting and then to the integro-differential systems.

\section{An outline of results in the finite-dimensional setting}

We start by recalling that a $C^1$-mapping
 $f:\mathbb{R}^n\to\mathbb{R}^n$ is locally invertible and its inverse
function is also $C^1$ provided that $\det f'(x) \neq 0$ for any
$x\in\mathbb{R}^n$. The last assumption means that $f'(x) $ is
invertible for any $x\in\mathbb{R}^n$, i.e.\
 is a linear mapping which is a bijection between $\mathbb{R}^n$ and
$\mathbb{R}^n$.
Let us first consider $C^1$-functions $f:\mathbb{R}\to\mathbb{R}$ such that
$f'(x) \neq 0$ for any $x\in\mathbb{R}$. Of course, such a function need not
be globally invertible as a simple
example of $\arctan $ shows. On the other hand, a $C^1$-function
$f(x) =x^{3}$ is globally invertible while its inverse is not $C^1$.
This function is not locally $C^1-$invertible around $0$. At the same
time, function $f(x) =x^{3}+x$ is locally (and globally)
invertible. We see that $| f_{i}(x) |\to \infty $ as $| x| \to \infty $ for
$i=1,2$. Such a property is called coercivity of a functional
$x\to | f(x) | $ and together with continuity it
says that this functional has at least one argument of a minimum. This is a
version of the celebrated Weierstrass Theorem since coercivity of
$| f| $ yields that for any $d\in \mathbb{R}$, the set $S^{d}$ is bounded, where
\[
S^{d}=\{u\in E:| f(u)| \leq d\}.
\]
Since continuity provides that such sets are closed, we are done.
In fact, it is easy to see that:
\begin{quote}
A locally invertible $C^1$-function
$f:\mathbb{R}\to\mathbb{R}$ (i.e.\ $f'(x) \neq 0$ for any
$x\in\mathbb{R}$) such that
$| f(x) | \to\infty $ as $| x| \to \infty $
is globally invertible, that is $f^{-1}:\mathbb{R}\to\mathbb{R}$
 is defined and $C^1$.
\end{quote}
It is easy to prove this result considering an auxiliary functional
$g:\mathbb{R}\to\mathbb{R}$ defined by
\[
g(x) =| f(x) -y| ^{2}
\]
for a fixed $y\in \mathbb{R}$. This functional is also coercive and contrary
to $x\to |f(x) | $, it is $C^1$. Thus for a fixed $y\in\mathbb{R}$
functional $g$ has an argument of a minimum $\overline{x}$ which by
Fermat's rule satisfies equation
\[
(f(\overline{x}) -y) f'(\overline{x}) =0
\]
and since $f'(x) \neq 0$ for any $x\in\mathbb{R}$ we see that
$f(\overline{x}) =y$ which solves the problem of $f $ being onto.
Now we concentrate on the question of $f$ being injective.
This is done by the help of Rolle's Theorem which, as Jean Mawhin points
out in his celebrated book \cite{Mawhin}, serves as an introductory Mountain
Pass Theorem. Thus assume that for some $y$ there exist $x_1,x_2$ such
that $f(x_1) =f(x_2) =y$. We see that for
functional $g$ it holds that $g(x_1) =g(x_2) =0$. Thus, by Rolle's Theorem
on interval $(x_1,x_2) $ we
must have at least one point $x_3$ such that $g'(x_3) =0$. This point, with the aid of the Weierstrass Theorem, can
be chosen as a point of a maximum over $[ x_1,x_2] $. Thus we
see that $g(x_3) >0$ and on the other hand, by the Fermat rule
$g(x_3) =0$ since $g'(x_3) =0$ implies $g(x_3) =0$. A contradiction completes
this simple proof.

We are aware that the above argumentation is far too sophisticated for a
single variable setting since under  the above assumptions it is strictly
monotone. But this very proof has the advantage of a generalization to maps
between Euclidean spaces and further between Banach spaces, provided that we
properly understand how the Weierstrass Theorem is generalized and what are
the analogies between all notions involved. Moreover, the additional
assumption that $\| f(x) \| \to \infty $
as $\| x\| \to \infty $ answers the following
question for a $C^1-$mapping $f:\mathbb{R}^n\to \mathbb{R}^n$:
\begin{quote}
What assumptions should be imposed on the mapping $f$  to
become a global diffeomorphism from local one?
\end{quote}
As for the known result we recall the Hadamard Theorem, see
\cite[Theorem 5.4]{jabri}.

\begin{theorem}\label{MainTheo copy(2)}
Let $X$, $Y$ be finite dimensional Euclidean spaces.
Assume that $f:X\to Y$ is a $C^1$-mapping such that
\begin{itemize}
\item $f'(x)$ is invertible for any $x\in X$,

\item  $\| f(x) \| \to \infty $ as $
\| x\| \to \infty $.
\end{itemize}
Then $f$ is a diffeomorphism.
\end{theorem}

The proof this theorem uses the approach  sketched above. Of course, Rolle's Theorem cannot be used here.  Instead the following three critical
points theorem is applied.

\begin{theorem}[{Finite-Dimensional MPT, Courant \cite[Theorem 5.2]{jabri}}]
Suppose that a $C^1$-functional $f:\mathbb{R}^n\to\mathbb{R}$ is coercive
and possesses two distinct strict relative minima $x_1$ and
$x_2$. Then $f$ possesses a third critical point $x_3$, distinct from
$x_1$ and $x_2$, which is not a relative minimizer, that is, in every
neighborhood of $x_3$, there exists a point $x$ such that $f(x) <f(x_3) $.
\end{theorem}

The above theorem inspired Idczak, Skowron and Walczak \cite{SIW} to
obtain a version of it for mappings between a Banach and a Hilbert space,
and it further inspired us to investigate if a Hilbert space can be replaced
by a Banach space \cite{GGS,galewskikoniorczyk}. In all this papers
a Mountain Pass Theorem is used which requires some refined
estimation of auxiliary functional around $0$. In this work we aim to follow
strictly the pattern described above. Since we see that the points $x_1$
and $x_2$ which we use for contradiction are in fact points of strict
local minima (otherwise we arrive at a contradiction with a local
diffeomorphism), we cannot use the classical Mountain Pass Theorem but 
instead the three critical points version due to Pucci and Serrin \cite{pucci}.

We would like to make some remarks at the end of this section. The global
inversion result which we have described has also a simple application as
the solvability tool of nonlinear equations of the following type
\[
f(x) =y
\]
where $f$ satisfies the mentioned assumptions. Since the solvability part is
reached via the Weierstrass Theorem and requires that some auxiliary
functional should have an argument of a minimum, it follows that a lower
semicontinuous weakly differentiable functional which has invertible weak
derivative would be sufficient. Thus we have the following easy tool:

\begin{corollary}
Assume that $f:\mathbb{R}^n\to\mathbb{R}^n$ is a weakly differentiable
mapping such that $\det f'(x) \neq 0$ for any $x\in\mathbb{R}^n$.
Assume that functional $g:\mathbb{R}^n\to\mathbb{R}$ defined by
$g(x) =\| f(x) \| ^{2}$ is lower semicontinuous and coercive.
Then for any fixed $y\in\mathbb{R}^n$ equation $f(x) =y$ has at least one solution.
\end{corollary}

We see that by our assumption, functional
$g(x) :=\|f(x) -y\| ^{2}$ is lower semicontinuous and coercive for any fixed
$y\in\mathbb{R}^n$. Thus the assertion follows easily. This result has no
meaning in case of $n=1$ since both types of differentiability coincide there.

\section{Some remarks on the infinite setting}

It is well known that, as in the case of $\mathbb{R}^n$ for $n\geq 2$,
there are two basic concepts of differentiability for
operators and functionals. Let $X,Y$ be Banach spaces, and assume that $U$
is an open subset of $X$. A mapping $f:U\to Y$ is said to be
G\^{a}teaux differentiable at $x_0\in U$ if there exists a continuous linear
operator $f_{G}'(x_0):X\to Y$ such that for every $h\in X$
\[
\lim_{t\to 0}\frac{f(x_0+th)-f(x_0)}{t}=f_{G}'(x_0)h.
\]
The operator $f_{G}'(x_0)$ is called the G\^{a}teaux derivative
of $f$ at $x_0$. We will denote it in the sequel by $f'$.

An operator $f:U\to Y$ is said to be Fr\'echet differentiable at $
x_0\in U$ if there exists a continuous linear operator $f'(x_0):X\to Y$
such that
\[
\lim_{\| h\| \to 0}\frac{\| f(x_0+h)-f(x_0)
-f'(x_0)h\| }{\| h\| }=0.
\]
The operator $f'(x_0)$ is called the Fr\'{e}chet derivative of
operator $f$ at $x_0$. When $F$ is Fr\'echet differentiable it is also
continuous and G\^{a}teaux differentiable. A mapping $f$ is continuously
Fr\'echet differentiable if $f':X\ni x\mapsto f'(x)\in
L(X,Y) $ is continuous in the respective topologies. If $f$ is
continuously G\^{a}teaux differentiable then it is also continuously
Fr\'echet differentiable and thus it is called $C^1$. The most common way to
prove the Fr\'{e}chet differentiability is that one shows that $f$ is
continuously G\^{a}teaux differentiable. In fact, for critical point theory
tools either the functional usually must be $C^1$ or locally Lipschitz and
so it is no surprise that G\^{a}teaux differentiability is only an auxiliary
tool.

A continuously Fr\'echet differentiable map $f:X\to Y$ is called a
diffeomorphism if it is a bijection and its inverse $f^{-1}:Y\to X$
is continuously Fr\'echet differentiable as well. Recalling the Inverse
Function Theorem, a continuously Fr\'echet differentiable mapping $
f:X\to Y$ such that for any $x\in X$ the derivative is surjective,
i.e. $f'(x)X=Y$ and invertible, i.e. there exists a constant
$\alpha _{x}>0$ such that
\[
\| f'(x)h\| \geq \alpha _{x}\| h\|
\]
defines a local diffeomorphism. We will write that
$f'(x) \in \operatorname{Isom}(X,Y) $ for mappings with such properties. When
$f'(x) \in \operatorname{Isom}(X,Y) $ for each $x\in X$, this means that for each point
$x$ in $X$, there exists an open set $U$
containing $x$, such that $f(U)$ is open in $Y$ and
$f\big|_{U}:U\to f(U)$ is a diffeomorphism. If $f$ is a diffeomorphism
it obviously defines a
local diffeomorphism. Thus the main problem to be overcome is to make a
local diffeomorphism into a global one.

Let $E$ be a Banach space. A Gateaux differentiable functional
$J:E\to \mathbb{R}$ satisfies the Palais-Smale condition if every
sequence $(u_{n})$ such that $(J(u_{n})) $ is bounded and
$J'(u_{n})\to 0$, has a convergent subsequence. The
following links the Palais-Smale condition with coercivity \cite{jabri}.

\begin{proposition}
Assume that $J\in C^1(E,R)$ is bounded from below and satisfies (PS)
condition. Then $J$ is coercive.
\end{proposition}

The \emph{converse statement} (coercivity implying the Palais-Smale
condition) is valid only in a finite-dimensional space. We note that with
the aid of the Palais-Smale condition, we can formulate a version of the
critical point theorem which serves as a counterpart of a direct method in
the calculus of variation. Namely, it can be applied when a functional is
not weakly l.s.c. or else we can consider it as some generalization of the
Weierstrass Theorem.

\begin{theorem}[{\cite[Proposition 10.1]{jabri}}] \label{CPT}
Let $E$ be a Banach space and $J:E\to \mathbb{R}$ a $C^1$-functional
which satisfies the Palais-Smale condition. Suppose in addition that $J$
is bounded from below. Then the infimum of $J$ is achieved at some point
$u_0\in E$ and $u_0$ is a critical point of $J$, i.e. $J'(u_0)=0$.
\end{theorem}

Now we have enough background to present the main result from \cite{SIW}.
They used the Mountain Pass Lemma (which can be cited for example in the
following form).

\begin{theorem}[Mountain Pass Theorem] \label{MPT_1}
 Let $E$ be a Banach space and assume that $J\in C^1(E, \mathbb{R})$
satisfies the Palais-Smale condition. Assume that
\begin{equation}
\inf_{\| x\| =r}J(x)\geq \max \{J(0),J(e)\}, \label{estimsphere}
\end{equation}
where $0<r<\| e\| $ and $e\in E$. Then $J$ has a non-zero
critical point $x_0$. If moreover, $\inf_{\| x\|=r}J(x)>\max \{J(0),J(e)\}$,
then  $x_0\neq e$.
\end{theorem}

In \cite{SIW}, using ideas contained in the proof of Theorem
\ref{MainTheo copy(2)} (see again \cite{jabri} for some nice version of the
proof or  our introductory remarks), the authors proved the following result  
concerning diffeomorphism between a Banach and a Hilbert space.

\begin{theorem} \label{TheoSW}
Let $X$ be a real Banach space and $H$ a real Hilbert space.
If $f:X\to H$ is a $C^1$-mapping such that:
\begin{itemize}
\item for any $y\in H$ the functional $\varphi :X\to \mathbb{R}$
given by the formula
\[
\varphi (x) =\frac{1}{2}\| f(x)
-y\| ^{2}
\]
satisfies the Palais-Smale condition,

\item  for any $x\in X$, $f'(x)\in \operatorname{Isom}(X,H) $,
\end{itemize}
then $f$ is a diffeomorphism.
\end{theorem}

The question arises  whether the Hilbert space $H$ in the
formulation of the above theorem can be replaced by a Banach space.
 This question is of some importance since one would expect diffeomorphism to act
between two Hilbert spaces or else two Banach spaces rather than between a
Hilbert and a Banach space. The applications given in \cite{SIW} work when
both $X$ and $H$ are Hilbert spaces.
 We see that given a Hilbert space $H$, the relation
$x\mapsto \frac{1}{2}\| x\| ^{2}$ can be
treated as $x\mapsto \frac{1}{2}\langle x,x\rangle $, where
$\langle \cdot ,\cdot \rangle $ is the scalar product. The other
point of view is to treat $x\mapsto \frac{1}{2}\| x\| ^{2}$
as a potential of a duality mapping between $H$ and $H^{\ast }$ and then
look at the composition of identity with some $C^1$-functional which is zero
only at $0$ and with derivative sharing the same property. These
observations allow us to easily generalize the mentioned result to a more
general setting.

\subsection{An approach by an auxiliary functional}

The version of the Mountain Pass Theorem which we use is sort of a
counterpart of a three critical points theorem which is actually what is
applied for the proof of Theorem \ref{MainTheo copy(2)}.

\begin{theorem}[{\cite[Theorem 2]{pucci}}] \label{theorem_pucci}
Let $X$ be a Banach space and let $J:X\to \mathbb{R}$ be a $C^1$-functional
satisfying the Palais-Smale condition with $0_{X}$ its strict local minimum.
If there exists $e\neq 0_{X}$ such that $J(e)\leqslant J(0_{X})$, then there is a
critical point $\bar{x}$ of $J$, with $J(\bar{x})>J(0_{X})$, which is not a
local minimum.
\end{theorem}

These observations will lead us towards obtaining the counterpart of Theorem
\ref{TheoSW} in a Banach space setting as well as related implicit function
results in the spirit of our introductory remarks. The application of
Theorem \ref{MPT_1} requires making an estimate of the action functional on
the sphere centered at $0$, while the application of Theorem
\ref{theorem_pucci} yields that some points define local minima. But the latter
property is easy to check since this functional is nonnegative and
being zero at some set other that an isolated singleton means that the local
invertibility is violated. Our main result now reads as follows
and it is based on the main result from \cite{GGS} which we improve by
getting rid of one of the assumptions and by using a suitably shortened
proof.

\begin{theorem} \label{MainTheo}
Let $X$, $Y$ be real Banach spaces. Assume that
$f:X\to Y$ is a $C^1$-mapping, $\eta :Y\to\mathbb{R}_{+}$ is a
$C^1$-functional and that the following conditions hold
\begin{itemize}
\item[(A1)] $(\eta (x) =0 \Leftrightarrow x=0) $ and
$(\eta '(x) =0 \Leftrightarrow x=0)$;

\item[(A2)] for any $y\in Y$ the functional $\varphi :X\to\mathbb{R}$
given by the formula
\[
\varphi (x) =\eta (f(x) -y)
\]
satisfies the Palais-Smale condition;

\item[(A3)] $f'(x)\in \operatorname{Isom}(X,Y) $ for any $x\in X$.
\end{itemize}
Then $f$ is a diffeomorphism.
\end{theorem}

\begin{proof}
We follow some ideas from \cite{MMR} with the necessary
modifications. In view of the remarks made above, condition (A3) implies
that $f$ is a local diffeomorphism. Thus it is sufficient to show that $f$
is onto and one-to-one.

First, we show that $f$ is onto.
Let us fix any point $y\in Y$.
Observe that $\varphi $ is a composition of two $C^1-$mappings, thus
$\varphi \in C^1(X,\mathbb{R})$. Moreover, $\varphi $ is bounded from
 below and it satisfies the
Palais-Smale condition. Thus it follows from Theorem \ref{CPT} that there
exists at least one argument of a minimum which we denote by $\overline{x}$.
We see by the chain rule for Fr\'{e}chet derivatives and by Fermat's
Principle that
\[
\varphi '(\overline{x})=\eta '(f(\overline{x}
) -y)\circ f'(\overline{x})=0.
\]
Since by (A3) mapping $f'(\overline{x})$ is invertible,
we see that $\eta '(f(\overline{x}) -y)=0$. Now it
follows by (A1) that
\[
f(\overline{x}) -y=0.
\]
Thus $f$ is surjective.

Now we argue by contradiction that $f$ is one-to-one.
Suppose there are $x_1$ and $x_2$, $x_1\neq x_2$, $x_1$, $x_2\in X$,
such that $f(x_1) =f(x_2) $. We will apply
Theorem \ref{theorem_pucci}. Thus we put $e=x_1-x_2$ and define
functional $\psi :X\to \mathbb{R}$ by the  formula
\[
\psi (x) =\eta (f(x+x_2) -f(x_1)) .
\]
By (A2), functional $\psi $ satisfies the Palais-Smale condition. Next, we
see by a direct calculation that $\psi (e) =\psi (0)=0$. Moreover, $0$
is a strict local minimum of $\psi $, since otherwise, in
any neighbourhood of $0$ we would have a nonzero $x$ with $f(x+x_2) =f(x_1) $
and this would contradict the fact
that $f$ defines a local diffeomorphism. Thus by Theorem \ref{theorem_pucci}
we note that $\psi $ has a critical point $v$ such that
$\psi (v) >0$. Since $v$ is a critical point we have
\[
\psi '(v)=\eta '(f(v+x_2) -f(x_1) )\circ f'(v+x_2)=0.
\]
Since $f'(v+x_2)$ is invertible, we see that $\eta '(f(v+x_2) -f(x_1) )=0$.
So by the assumption
(A1) we calculate $f(v+x_2) -f(x_1) =0$. This
means that $\psi (v)=0$ which is impossible.
\end{proof}

Next we state our result with some remarks.

\begin{remark} \rm
We see that by putting $\eta (x) =\frac{1}{2}\|x\| ^{2}$ we easily obtain
Theorem \ref{TheoSW} from Theorem \ref{MainTheo}.
 Moreover, in \cite{GGS} the following condition is additionally assumed:
\begin{itemize}
\item[(A4)] there exist positive constants $\alpha $, $c$, $M$ such that
\[
\eta (x) \geq c\| x\| ^{\alpha }\quad \text{for }\| x\| \leq M.
\]
\end{itemize}
We do not need this condition now since we use a different tool for the
proof the main result. Moreover, the proof becomes considerably simpler and
the result is now a full counterpart of the Hilbert space setting case.
\end{remark}

We conclude with an example of a functional $\eta $ satisfying our
assumptions.
Let us take for $p\geq 2$ a uniformly convex Banach space
\[
W_0^{1,p}([0,1],\mathbb{R}):=W_0^{1,p}
\]
consisting of absolutely continuous functions $x:[0,1]\to \mathbb{R}$
such that $x(0)=x(1)=0$ and $x'\in L^p([ 0,1] ,\mathbb{R})$ considered
with the usual norm
\[
\|x\|_{W_0^{1,p}}=(\int_0^1| x'(t)|^pdt)^{1/p},\quad x\in W_0^{1,p}.
\]
We also refer to \cite{BrezisBook} for some background on this space.

\begin{lemma} \label{lem_dif_LP}
 Assume that $2\leq p<+\infty $. A functional $f:L^p\to \mathbb{R}$ given
by the formula $f(u) =\frac{1}{p}\int_0^1|u(s)|^pds=\frac{1}{p}\| u\|_{L^p}^p$
is continuously differentiable and for any $u\in L^p$ we
have
\[
f'(u)v=\int_0^1| u(t)| ^{p-2}u(t)v(t)dt\text{
for }v\in L^p.
\]
\end{lemma}

\begin{remark} \label{rem_diff_wp} \rm
For the case $p=2$ the assertion of Lemma \ref{lem_dif_LP}
follows from the properties of the scalar product. For the case $2<p<+\infty$,
it follows from Lemma \ref{lem_dif_LP} and from the chain rule, that a
functional $h:W_0^{1,p}\to\mathbb{R}$ defined by
\begin{equation}
h(u) =\frac{1}{p}\int_0^1| u'(t)| ^pdt  \label{def_fi}
\end{equation}
is continuously differentiable and for any $u\in W_0^{1,p}$ we have
\[
h'(u)v=\int_0^1| u'(t)|
^{p-2}u'(t)v'(t)dt\text{ for }v\in W_0^{1,p}.
\]
\end{remark}

Thus we see that both $f$ and $h$ above provide us with examples.

\subsection{An approach by a duality mapping}

In this section we improve results from \cite{galewskikoniorczyk} by
extending them to cover the case of any duality mapping relative to some
increasing function and not to the special function
$t\to |t| ^{p-1}$ for $p>1$, as provided therein. The second improvement
is the simplification of the proof in the spirit described above.

A normed linear space $E$ is called strictly convex if the unit sphere
contains no line segments on its surface, i.e., condition
$\|x\| =1,\| y\| =1,x\neq y$ implies that
\[
\| \frac{1}{2}(x+y) \| <1
\]
or in other words, that the unit sphere is a strictly convex set.
The space $E$ is called uniformly convex, if for each
$\varepsilon \in (0,2]$ there exists $\delta (\varepsilon ) >0$ such that if
$\| x\| =1,\| y\| =1$ and $\|x-y\| \geq \varepsilon $, then
$\| x+y\| \leq 2(1-\delta (\varepsilon ) ) $.
A uniformly convex space is necessarily strictly convex and
reflexive.

We recall from \cite{MahJeb} the notion of duality mapping from $E$ to
$E^{\ast }$ relative to a normalization function. In the sequel we shall
write duality mapping with the understanding that we mean a duality mapping
relative to some normalization function. A continuous function
$\varphi :\mathbb{R}_{+}\to \mathbb{R}_{+}$ is called a normalization function
if it is strictly increasing, $\varphi (0) =0$ and
$\varphi(r)\to \infty $ with $r\to \infty $. A duality mapping on
$E$ corresponding to a normalization function $\varphi $ is an operator
$A:E\to 2^{E^{\ast }}$ such that for all $u\in E$ and
$u^{\ast }\in A(u) $,
\[
\| A(u) \| _{\ast }=\varphi (\|u\| ) , \quad
\langle u^{\ast },u\rangle
=\| u^{\ast }\| _{\ast }\| u\| .
\]
Some remarks are in order, especially those concerning the assumptions
on a duality mapping.  We recall from \cite{MahJeb} that
\begin{itemize}
\item[(i)] for each $u\in E$, $A(u) $ is a bounded, closed and convex
subset of $E^{\ast };$

\item[(ii)] $A$ is monotone
\[
\langle u_1^{\ast }-u_2^{\ast },u_1-u_2\rangle \geq
(\varphi (\| u_1\| ) -\varphi (\| u_2\| ) ) (\|
u_1\| -\| u_2\| )
\]
for each $u_1,u_2\in E$, $u_1^{\ast }\in A(u_1),u_2^{\ast }\in A(u_2)$;

\item[(iii)] for each $u\in E$, $A(u) =\partial \psi (u) $,
where $\partial \psi (\cdot ):E\to 2^{E^{\ast }}$ denotes the
subdifferential in the sense of convex analysis of the functional $\psi
(u) =\int_0^{\| u\| }\varphi (t) dt$, i.e.
\[
\partial \psi (u) =\big\{ u^{\ast }\in E^{\ast }:\psi (
y) -\psi (u) \geq \langle u^{\ast },y-u\rangle
_{E^{\ast },E}\text{ for all }y\in E\big\} ;
\]

\item[(iv)] if $E^{\ast }$ is strictly convex, then $card(A(u) )=1$,
for all $x\in E$;

\item[(v)] when $E$ is reflexive and strictly convex, then operator
$A:E\to E^{\ast }$ is demicontinuous, which means that if
$x_{n}\to x$ in $E$ then $A(x_{n}) \rightharpoonup A(x) $ in $E^{\ast }$.
\end{itemize}
We see by \cite[Proposition 2.8]{phelps} that since $A$ is
demicontinuous, we obtain that functional $\psi $ is differentiable in the
sense of G\^{a}teaux and operator $A$ being its derivative.

\begin{lemma} \label{lem_duality_mapping}
Assume that $E$ is a reflexive Banach space with
a strictly convex dual $E^{\ast }$. Then the duality mapping
$A:E\to E^{\ast }$\ corresponding to a normalization function $\varphi $ is
single-valued and the functional $\psi $ is differentiable in the sense of
 G\^{a}teaux with $A$ being its G\^{a}teaux derivative.
\end{lemma}

\begin{proof}
The duality mapping $A$ is now single valued and its
potential, i.e. $\psi $, has a single-valued subdifferential which is
demicontinuous. Now the argument given prior to the formulation finishes the
proof.
\end{proof}

In view of Lemma \ref{lem_duality_mapping}, it is apparent that assuming the
continuous differentiability of a potential of a duality mapping is not a
very restrictive condition. Indeed, in order to get some reasonable result
we will have to assume that the duality mapping has the potential which is
continuously differentiable. This is necessary in order to obtain that the
functional $h:X\to \mathbb{R}$ given by the formula
\[
h(x) =\psi (f(x) -y)
\]
is continuously differentiable. Moreover, functional $\psi $ can be
considered as a potential of a duality mapping $A$ in case it is
single-valued. In this case we write $A:Y\to Y^{\ast }$ and by
writing $A:Y\to Y^{\ast }$ we implicitly assume that $A$ is single
valued. We will follow this observation in the sequel. We formulate our
second global diffeomorphism result.

\begin{theorem} \label{GLOBIFnew_v1}
Let $X$, $Y$ be a real Banach spaces. Let the potential
$\psi $ of a duality mapping $A:Y\to Y^{\ast }$ corresponding to a
normalization function $\varphi $ be continuously G\^{a}teaux
differentiable. Assume that $f:X\to Y$ is a $C^1$-mapping such
that:
\begin{itemize}
\item[(A5)] for any $y\in Y$ the functional $h:X\to \mathbb{R}$ given by
the formula
\[
h(x) =\psi (f(x) -y)
\]
satisfies the Palais-Smale condition,

\item[(A6)] $f'(x)\in \operatorname{Isom}(X,Y) $ for any $x\in X$.
\end{itemize}
Then $f$ is a diffeomorphism.
\end{theorem}

\begin{proof}
Let us fix a point $y\in Y$. Functional $h$ is a composition
of two $C^1-$mappings, so it is $C^1$ itself. Moreover, $h$ is bounded
from below and it satisfies the Palais-Smale condition by (A5). Thus it
follows from Theorem \ref{CPT} that there exists an argument of a minimum
which we denote by $\overline{x}$. We see by the chain rule and Fermat's
Principle and by the assumptions on a duality mapping that
\[
0=h'(\overline{x})=A(f(\overline{x}) -y)
\circ f'(\overline{x}) .
\]
Since by (A6), mapping $f'(\overline{x}) $ is
invertible, we get that $A(f(\overline{x}) -y) =0$.
Now, by the property that $\| A(u) \| _{\ast
}=\varphi (\| u\| ) $ we note that
\[
\| A(f(\overline{x}) -y) \| _{\ast
}=\varphi (\| f(\overline{x}) -y\|
) .
\]
So it follows, since $\varphi (0) =0$ and since $\varphi $ is
strictly increasing, that
\[
f(\overline{x}) -y=0,
\]
which proves the existence of $x\in X$ for every $y\in Y$, such that $
f(\overline{x}) =y$. The uniqueness can be shown by
contradiction exactly as before.
\end{proof}

Now we provide a simple example of a space and a duality mapping for which
the assumptions of the above result hold.

Let us define a single valued operator $A:W_0^{1,p}\to (
W_0^{1,p}) ^{\ast }$ such that
\[
\langle Au,v\rangle =\int_0^1| u'(t)| ^{p-2}u'(t)v'(t)dt.
\]
It follows from Remark \ref{rem_diff_wp} that if $2\leq p<+\infty $, then $A$
is a potential operator and its $C^1-$potential is $h$ defined by
\eqref{def_fi}. Therefore by the cited remarks from \cite{MahJeb} we get that a
duality mapping on $W_0^{1,p}$ corresponding to a normalization function $
t\to t^{p-1}$ is defined by $A$.

\section{Applications}

\subsection{Applications to algebraic equations}

We conclude this paper with some applications to the unique solvability of
nonlinear equations of the form $Ax=F(x)$\ where $A$\ is a nonsingular
matrix and $F$\ is a $C^1$ nonlinear operator. Some motivation to extend
existence tools for such equations can be found in
\cite{add3,add4,add2,add1}.
We consider the  problem
\begin{equation}
Ax=F(x) ,  \label{equAPP}
\end{equation}
where $A$ is an $n\times n$ matrix (possibly singular) and
$F:\mathbb{R}^n\to\mathbb{R}^n$ is a $C^1-$mapping.
We consider $\mathbb{R}^n$ with the Euclidean norm in both, theoretical
results and in the example which follows.

To apply Theorem \ref{MainTheo copy(2)} to the solvability of
\eqref{equAPP} we need some assumptions. Let us recall that if $A^{\ast }$
denotes the transpose of matrix $A$, then $A^{\ast }A$ is symmetric and
positive semidefinite. Let $\delta _{\max }(A) $ denote the
greatest of the singular values of $A$, i.e. the eigenvalues of $A^{T}A$,
with the obvious meaning of $\delta _{\min }(A) $.  We assume
that $A$ is different from the $zero$ matrix, so that both mentioned values
are positive.

\begin{theorem}\label{firstAlgTheo}
Assume that $F:\mathbb{R}^n\to\mathbb{R}^n$ is a $C^1$-mapping and the
following conditions hold:
\begin{itemize}
\item[(i)] either there exists a constant $0<a<\delta _{\min }(A) $
such that
\[
\| F(x) \| \leq a\| x\|
\]
for all $x\in\mathbb{R}^n$ with sufficiently large norm,
or  else there exists a constant
$b>\delta _{\max }(A) $ such that
\[
\| F(x) \| \geq b\| x\|
\]
for all $x\in \mathbb{R}^n$ $x\in\mathbb{R}^n$ with sufficiently large norm;

\item[(ii)] $\det (A-F{'}(x) ) \neq 0$ for every
every $x\in \mathbb{R}^n$.
\end{itemize}
Then Problem \eqref{equAPP} has exactly one solution for any
$\xi \in\mathbb{R}^n$.
\end{theorem}

\begin{proof}
We put $\varphi (x) =Ax-F(x) $. From
the first possibility of assumption (ii) it follows for
$x\in\mathbb{R}^n$ with sufficiently large norm,
\begin{align*}
\| \varphi (x) \|
&=\| Ax-F( x) \| \geq \| Ax\| -\| F(x) \|  \\
&\geq \sqrt{\langle A^{\ast }Ax,x\rangle }-a\|x\|
 \geq (\delta _{\min }(A) -a) \| x\| .
\end{align*}
Hence the function $\varphi $ is coercive. Since $\det \varphi '(x) \neq 0$
for every  $x\in\mathbb{R}^n$, it follows by 
Theorem \ref{MainTheo copy(2)} that $\varphi $ is a global
homeomorphism and thus equation \eqref{equAPP} has exactly one solution for
any $\xi \in\mathbb{R}^n$. The second case of assumption (ii)
follows likewise.
\end{proof}

\begin{remark} \rm
We note that  to obtain coercivity of the function $\varphi $ in the above
theorem we can employ the following assumption instead of (ii),
\begin{itemize}
\item[(iia)] either there exist constants $\alpha >0$, $0<\gamma <1$ such that
\[
\| F(x) \| \leq \alpha \| x\|
^{\gamma }
\]
for all $x\in\mathbb{R}^n$ with sufficiently large norm, or

\item[(iib)]  there exist constants $\beta >0$, $\theta >1$ such that
\[
\| F(x) \| \geq \beta \| x\|^{\theta }
\]
for all $x\in\mathbb{R}^n$ with sufficiently large norm.
\end{itemize}
\end{remark}

Now we provide some examples of problems which we can consider by the above
method.

\begin{example} \rm
Consider the indefinite matrix
\[
A=\begin{bmatrix}
-2 &  1 \\
 6 & -3
\end{bmatrix}
\]
and the function $F:\mathbb{R}^{2}\to\mathbb{R}^{2}$ given by
\[
F(x,y) =(x^{3}+y+1,6x+y+y^{3}+1)\,.
\]
On $\mathbb{R}^{2}$ consider the Euclidean norm,
$\| (x,y) \| = \sqrt{x^{2}+y^{2}}$.
We recall that $\| (x,y)\| \leq 2^{\frac{1}{3}}\sqrt[6]{x^{6}+y^{6}}$.
Note that $F(x,y) =(x^{3},y^{3}) +(0,6x) +(y,y)+(1,1)$. Let
\[
\varphi (x,y) =F(x,y)-A(x,y),\quad (x,y)\in \mathbb{R}^{2}.
\]
Hence
\begin{align*}
\| \varphi (x,y) \|
&\geq \frac{1}{2}\| (x,y)\| ^{3}-6\sqrt{2}\| (x,y)\| -\sqrt{2}-\| A\| \| (x,y)\|   \\
&=\| (x,y)\| \Big(\frac{1}{2}\| (x,y)\|
^{2}-(6\sqrt{2}\mathbf{+}\| A\| ) -\frac{\sqrt{2}}{\| (x,y)\| }\Big) .
\end{align*}
From the above sequence of inequalities it follows that $\varphi $ is
coercive.

One can easily see that for any $(x,y) \in\mathbb{R}^{2}$
\[
F'(x,y) -A=\begin{bmatrix}
3x^{2}+3 & 0 \\
0 & 3y^{2}+4
\end{bmatrix}.
\]
Since  $\det (F'(x,y) -A) >0$, we see
that the problem $Ax=F(x) $ has exactly one (nontrivial) solution.
\end{example}

\subsection{Application to an integro-differential system}

In this section we propose some improvement of results from
\cite{galewskikoniorczyk} as far as the growth assumptions and methods of the
proof are concerned. Namely, we use the Bielecki type norm on the underlying
space instead of the regular one. Since the proofs do not differ that much
apart from some estimation, we provide only the main differences referring to
\cite{galewskikoniorczyk} for the more detailed reasoning. We were inspired
by \cite{majewski} to come up with these results.

Prior to formulating the problem under consideration we introduce some
required function space setting. We introduce
\[
W^{1,p}([0,1],\mathbb{R}^n)
=\big\{x:[0,1]\to \mathbb{R}^n\text{ is
absolutely continuous, }x'\in L^p([0,1],\mathbb{R}^n)\big\}.
\]
Here $x'$ denotes the a.e. derivative of $x$. Further, we
denote $L^p([0,1],\mathbb{R}^n)$ by $L^p$ and
$W^{1,p}([0,1],\mathbb{R}^n)$ by $W^{1,p}$. The $W^{1,p}$ space
is equipped with the
usual norm $\|x\|_{W^{1,p}}^p=\|x\|_{L^p}^p+\|x'\|_{L^p}^p$.
The Sobolev space  is defined as
\[
\tilde{W}_0^{1,p}([0,1],\mathbb{R}^n)=\{x\in W^{1,p},x(0)=0\}
\]
and it is equipped with the norm
\begin{equation}
\|x\|_{\tilde{W}_0^{1,p}}
=\Big(\int_0^1| x'(t)| ^pdt\Big)^{1/p},\quad x\in \tilde{W}_0^{1,p}
\label{W01p_norm}
\end{equation}
equivalent to $\|x\|_{W^{1,p}}$. By definition, for any $p>1$, we have the
following chain of embeddings
\begin{equation}
\tilde{W}_0^{1,p}\hookrightarrow W^{1,p}\hookrightarrow L^p.
\label{space_imbedding}
\end{equation}
There exists a constant $C$ such that for any $u\in \tilde{W}_0^{1,p}$
\[
\|u\|_{W^{1,p}[0,1]}\leq C\|u'\|_{L^p[0,1]}.
\]
The space $\tilde{W}_0^{1,p}$ is uniformly convex.

In the literature, the existence of the solution to integro-differential
equation is obtained by the Banach fixed point theorem or another type of
fixed point theorem, see \cite{integro1,integro2}.

Let us formulate a nonlinear integro-differential equation with variable
integration limit with an initial condition, which reads as follows
\begin{gather}
x'(t)+\int_0^{t}\Phi (t,\tau ,x(\tau ))d\tau =y(t),\quad
\text{for a.e. }t\in [ 0,1],  \label{I_D_eq} \\
x(0)=0,  \label{I_D_iv}
\end{gather}
where $y\in L^p$ is fixed for the time being.

Now we impose assumptions on the nonlinear term. These ensure that the
problem is well posed in the sense that the solution to
\eqref{I_D_eq}-\eqref{I_D_iv} exists, it is unique and the solution operator
depends in a differentiable manner on a parameter $y$ provided we allow it to vary.
This implies that problem \eqref{I_D_eq}-\eqref{I_D_iv} is well posed in the
sense of Hadamard.

Let $P_{\Delta }=\{(t,\tau )\in [ 0,1]\times [ 0,1];\tau \leq t\} $.
We assume that function $\Phi :P_{\Delta }\times \mathbb{R}
^n\to \mathbb{R}^n$ satisfies the following conditions:
\begin{itemize}
\item[(A7)] 
$\Phi (\cdot ,\cdot ,x)$ is measurable on $
P_{\Delta }$ for any $x\in \mathbb{R}^n$ and $\Phi (t,\tau ,\cdot )$ is
continuously differentiable on $\mathbb{R}^n$ for a.e. $(t,\tau )\in
P_{\Delta }$;

\item[(A8)] 
there exist functions $a$, $b\in L^p(P_{\Delta},R_0^{+})$ such that
\[
|\Phi (t,\tau ,x)|\leq a(t,\tau )|x|+b(t,\tau )
\]
for a.e. $(t,\tau )\in P_{\Delta }$, all $x\in \mathbb{R}^n$ and there
exists a constant $\overline{a}>0$ such that
\[
\int_0^{t}a^p(t,\tau )d\tau \leq \overline{a}^p
\]
for a.e. $t\in [ 0,1]$.

\item[(A9)] 
there exists functions $c\in L^p(P_{\Delta },\mathbb{R}_0^{+})$,
$\alpha \in C(\mathbb{R}_0^{+},\mathbb{R}_0^{+})$
and a constant $C>0$ such that
\[
|\Phi _{x}(t,\tau ,x)|\leq c(t,\tau )\alpha (|x|)
\]
for a.e. $(t,\tau )\in P_{\Delta }$ and all $x\in \mathbb{R}^n$; moreover
\[
\int_0^{t}c^{q}(t,\tau )d\tau \leq C,\text{ for }a.e.\text{ }t\in
[ 0,1].
\]
\end{itemize}

\begin{remark} \rm
In \cite{galewskikoniorczyk} it was assumed that
\[
\|a\|_{L^p(P_{\Delta },\mathbb{R})}<2^{-\frac{(p-1)}{p}}
\]
which considerably restricts the growth.
\end{remark}

 For any $k>0$ let us define another form of the Bielecki type norm
\begin{equation}
\|x\|_{\tilde{W}_0^{1,p},k}
=\Big(\int_0^1e^{-kt}|x'dt\Big) ^{1/p}.  \label{Bielecki_1_def}
\end{equation}
For $k=0$ the above function defines a norm introduced by \eqref{W01p_norm}
and therefore hereafter we will skip index $0$. It is easy to notice that
\begin{equation}
e^{-k/p}\|x\|_{\tilde{W}_0^{1,p}}\leq \|x\|_{\tilde{W}
_0^{1,p},k}\leq \|x\|_{\tilde{W}_0^{1,p}}  \label{rel_equ_biel}
\end{equation}
For any $k>0$ and $x\in \tilde{W}_0^{1,p}$ we assert the following
relations:
\begin{gather}
\|x\|_{k}\leq \frac{\|x\|_{\tilde{W}_0^{1,p},k}}{k^{1/p}}, \label{Bielecki_1_1}\\
\| \int_0^{\cdot }|x(\tau )|d\tau \| _{k}
=\Big(\int_0^1e^{-kt}\Big(\int_0^{t}|x(\tau )|d\tau \big) ^pdt\Big)
^{1/p}\leq \frac{\|x\|_{\tilde{W}_0^{1,p},k}}{k^{2/p}}
\label{Bielecki_1_2}
\end{gather}
where the symbol $\int_0^{\cdot }u(\tau )d\tau $ denotes the function
$[0,1]\ni t\to \int_0^{t}u(\tau )d\tau $. Now let us prove the
stated relations, starting with \eqref{Bielecki_1_1}. Fix $k>0$ and
$x\in\tilde{W}_0^{1,p}$:
\begin{align*}
\|x\|_{k}^p
&=\int_0^1e^{-kt}|x(t)|^pdt
 =\int_0^1e^{-kt}|\int_0^{t}x'(\tau )d\tau | ^pdt \\
&\leq \int_0^1e^{-kt}\int_0^{t}| x'(\tau )|^pd\tau dt
 =  \int_0^1| x'(\tau )| ^p(\int_{\tau }^1e^{-kt}dt) d\tau \\
&=\frac{1}{k}\int_0^1e^{-kt}|x'^pdt-\frac{e^{-k}}{k}\int_0^1|x'^pdt \\
&\leq   \frac{1-e^{-k}}{k}\int_0^1e^{-kt}|x'^pdt
\leq \frac{ \|x\|_{\tilde{W}_0^{1,p},k}^p}{k}.
\end{align*}
Now let us turn to the relation \eqref{Bielecki_1_2}:
\begin{align*}
 \|\int_0^{\cdot }|x(\tau )|d\tau \|_{k}^p
&=\int_0^1e^{-kt}(\int_0^{t}|x(\tau )|d\tau )^pdt
 \leq \int_0^1e^{-kt}(\int_0^{t}|x(\tau )|^pd\tau )dt \\
&= \int_0^1|x(\tau )|^p(\int_{\tau }^1e^{-kt}dt)d\tau \\
&=\frac{1}{k}\int_0^1e^{-kt}|x(t)|^pdt-\frac{e^{-k}}{k}
\int_0^1|x(t)|^pdt \\
&\leq   \frac{\|x\|_{k}^p}{k}\leq \frac{\|x\|_{\tilde{W}_0^{1,p},k}^p
}{k^{2}}.
\end{align*}
To apply Theorem \ref{MainTheo} we can define functional $\varphi :\tilde{W}
_0^{1,p}\to \mathbb{R}$ as follows
\begin{align*}
\varphi (x)
&=(1/p)\| f(x)-y\| _{k}^p \\
&=(1/p)\int_0^1e^{-kt}| x'(t)-y(t)+\int_0^{t}\Phi (t,\tau ,x(\tau ))d\tau | ^pdt.
\end{align*}
We can define functional $\varphi :\tilde{W}_0^{1,p}\to \mathbb{R}$
in the form
\begin{align*}
\varphi (x)
&=(1/p)\| f(x)-y\| _{k}^p \\
&=(1/p)\int_0^1e^{-kt}| x'(t)-y(t)+\int_0^{t}\Phi (t,\tau ,x(\tau ))d\tau | ^pdt.
\end{align*}
Having in mind the relation \eqref{rel_equ_biel}, which states that $L^p$
norm $\|\cdot \|_{L^p}$ and the Bielecki norm $\|\cdot \|_{k}$ are
equivalent, the following inequality can be deduced for any
$x\in \tilde{W}_0^{1,p}$:
\begin{align*}
(p\varphi (x)) ^{1/p}
&=\| x'(\cdot) -y(\cdot ) +\int_0^{\cdot }\Phi (\cdot ,\tau
,x(\tau ))d\tau \| _{k} \\
&\geq \|x'\|_{k}-\|y\|_{k}-\|\int_0^{\cdot }\Phi (\cdot
,\tau ,x(\tau ))d\tau \|_{k} \\
&\geq \|x'\|_{k}-\|y\|_{k}-\overline{a}\|
\int_0^{\cdot }x(\tau )d\tau \| _{k}-\|
\int_0^{\cdot }b(\cdot ,\tau )d\tau \| _{k} \\
&\geq \|x\|_{\tilde{W}_0^{1,p},k}-\frac{\overline{a}}{k^{2/p}}
 \|x\|_{\tilde{W}_0^{1,p},k}+d,
\end{align*}
where $d=\|y\|_{k}-\| \int_0^{\cdot }b(\cdot ,\tau )d\tau \| _{k}$.
 For sufficiently large $k>0$, that is $k>\max \{1,\overline{a}^{\frac{p}{2}}\}$,
 we have the coercivity of functional $\varphi$.

Using the above estimates and exactly the same arguments as in
\cite{galewskikoniorczyk}, we can prove the following result.

\begin{theorem}\label{appTheo_1}
Under the above assumptions, for any fixed $y$ $\in L^p$,
problem \eqref{I_D_eq}-\eqref{I_D_iv} has a unique solution
$x_{y}\in \tilde{W}_0^{1,p}$. Moreover, the operator
\[
L^p\ni y\to x_{y}\in \tilde{W}_0^{1,p}
\]
which assigns to each $y\in L^p$ a solution to \eqref{I_D_eq}-\eqref{I_D_iv},
is continuously differentiable.
\end{theorem}

We complete this section with an example of a nonlinear term satisfying our
assumptions (A7)--(A9).
Let us consider the function
$\Phi :P_{\Delta }\times \mathbb{R}\to \mathbb{R}$ defined by
\[
\Phi (t,\tau ,x)=\alpha (t-\tau )^{5/2}\ln (1+(t-\tau )^{2}x^{2})
\]
for $t,\tau \in [0,1]$, $t>\tau $, $x\in \mathbb{R}$, where $\alpha >0$
is fixed. Since $\ln (1+s^{2}z^{2})\leq |s|+|z|$ for $s,z\in \mathbb{R}$,
we see that
\[
|\Phi (t,\tau ,x)|\leq \alpha (t-\tau )^{5/2}|x|+\alpha (t-\tau )^{5/2}.
\]
Let us put
\[
a(t,\tau )=\alpha (t-\tau )^{5/2}
\]
for $t,\tau $ $\in [ 0,1]$, $t>\tau $. Then by a direct calculation we
obtain
\[
\|a\|_{L^p(P_{\Delta },\mathbb{R})}^p\leq \alpha ^p\frac{4}{(5p+2)(5p+4)}
=:\overline{a}.
\]
Moreover,
\begin{gather*}
|\Phi _{x}(t,\tau ,x)|\leq \alpha (t-\tau )^{5/2}|x|, \\
\int_0^{t}c(t,\tau )^{q}d\tau =2^{-p}\int_0^{t}(t-\tau )^{5q/2}d\tau =
\frac{2^{1-p}}{5q+2}t^{(5q/2)+1}\leq \frac{2^{1-p}}{5q+2}, \quad t\in [ 0,1].
\end{gather*}
Hence, $\Phi $ satisfies assumptions (A7)--(A9).
 Theorem \ref{MainTheo} shows that the initial-value problem
\[
x'(t)+\int_0^{t}2^{1-p}(t-\tau )^{1/2}\ln (1+(t-\tau
)^{2}x^{2})d\tau
=y(t),\quad \text{a.e. } t\in [ 0,1]
\]
has a unique solution $x_{y}\in \tilde{W}_0^{1,p}$ for any fixed
$y\in L^p$. Moreover, the mapping
\[
L^p\ni y\to x_{y}\in \tilde{W}_0^{1,p}
\]
is continuously differentiable.

\subsection*{Acknowledgements}
D. Repov\v{s}  was supported by the Slovenian Research Agency grants P1-0292,
N1-0083, N1-0064, J1-8131, and J1-7025.

\end{document}